\documentclass[11pt]{article}
\usepackage{caption}
\usepackage{epsf,epsfig,amsfonts,amsgen,amsmath,amstext,amsbsy,amsopn,amsthm}
\usepackage{color}
\usepackage{mathrsfs}
\usepackage{tikz}
\usepackage{enumerate}
\usepackage{url}
\usepackage{enumitem}
\usepackage[usenames,dvipsnames]{pstricks}
\usepackage{epsfig}
\usepackage{pst-grad} 
\usepackage{pst-plot} 
\usepackage[space]{grffile} 
\usepackage{etoolbox}
\usepackage{geometry}

\def\ex{\mathrm{ex}}

\newtheorem{thm}{Theorem}[section]

\newtheorem{lem}[thm]{Lemma}

\newcommand{\1}{{\uppercase\expandafter{\romannumeral1}}}
\newcommand{\2}{{\uppercase\expandafter{\romannumeral2}}}
\begin{document}
\title{On the maximum number of edges in $k$-critical graphs}
\author{
Cong Luo\thanks{School of Mathematical Sciences, University of Science and Technology of China, Hefei, Anhui 230026, China. Email: luoc@mail.ustc.edu.cn.}
~~~~~
Jie Ma\thanks{School of Mathematical Sciences, University of Science and Technology of China, Hefei, Anhui 230026, China. Research supported in part by the National Key R and D Program of China 2020YFA0713100, National Natural Science Foundation of China grant 12125106, and Anhui Initiative in Quantum Information Technologies grant AHY150200. Email: jiema@ustc.edu.cn.}
~~~~~
Tianchi Yang\thanks{Department of Mathematics, National University of Singapore, 119076, Singapore. Research supported in part by Professor Hao Huang's start-up grant at NUS and an MOE Academic Research Fund (AcRF) Tier 1 grant. Email: tcyang@nus.edu.sg.}}

\date{}
	
\maketitle

\begin{abstract}
A graph is called $k$-critical if its chromatic number is $k$ but any proper subgraph has chromatic number less than $k$.
An old and important problem in graph theory asks to determine the maximum number of edges in an $n$-vertex $k$-critical graph. 
This is widely open for any integer $k\geq 4$.
Using a structural characterization of Greenwell and Lov\'asz and an extremal result of Simonovits,
Stiebitz proved in 1987 that for $k\geq 4$ and sufficiently large $n$, this maximum number is less than the number of edges in the $n$-vertex balanced complete $(k-2)$-partite graph.
In this paper we obtain the first improvement on the above result in the past 35 years.
Our proofs combine arguments from extremal graph theory as well as some structural analysis.
A key lemma we use indicates a partial structure in dense $k$-critical graphs, which may be of independent interest. 
\end{abstract}

\section{Introduction}
All graphs we consider are finite and simple.
A graph $G$ is \emph{$k$-colorable} if we can assign $k$ colors to its vertices such that no adjacent vertices receive the same color.
We say a graph $G$ is  \emph{$k$-chromatic} if it is  $k$-colorable but not $(k-1)$-colorable.
A graph $G$ is called {\it $k$-critical} if $G$ is $k$-chromatic but every its proper subgraph is $(k-1)$-colorable.
For $k\in \{1,2\}$ the only $k$-critical graph is $K_k$, and the family of 3-critical graphs is precisely the family of odd cycles.
In this paper, we consider $k$-critical graphs for $k\ge 4$.
	
A central problem in graph theory asks to determine the maximum number of edges $f_k(n)$ in an $n$-vertex $k$-critical graph (see \cite{JT}).
Before we discuss the literature on $f_k(n)$, we would like to point out a relevant yet easy fact that the {\it Tur\'an graph} $T_k(n)$ (that is, the $n$-vertex balanced complete $k$-partite graph) has the maximum number of edges among all $n$-vertex $k$-chromatic graphs.
Dirac \cite{D} gave $f_6(n)\ge \frac14 n^2+n$ by considering the graphs obtained by joining two vertex-disjoint odd cycles with the same number of vertices.
Toft \cite{Toft70} proved that for every $k\ge 4$, there exists a positive constant $c_k$ such that $f_k(n)\ge c_k n^2$ holds for all integers $n\ge k$ (except $n=k+1$). 
In the most basic and interesting cases $k=4,5$, the constants are given by 
$$c_4\geq \frac{1}{16}=0.0625 \mbox{~~ and ~~} c_5\geq \frac{4}{31}\geq 0.129.$$
In the general case when $k\geq 6$, explicit constructions in \cite{Toft70} show that there exist infinitely many values of $n$ such that
$$f_k(n)\geq \left(\frac12-\frac{3}{2k-\delta_k}\right)n^2,$$
where $\delta_k=0$ if $k\equiv 0 \pmod 3$, $\delta_k=8/7$ if $k\equiv 1 \pmod 3$, and $\delta_k=44/23$ if $k\equiv 2 \pmod 3$.
To our best knowledge, no construction for giving better constants $f_k(n)/n^2$ have been found since.
It is also an open question if $\lim_{n\to \infty} \frac{f_k(n)}{n^2}$ exits for any $k\geq 4$.
In 2013, Pegden \cite{P} considered dense triangle-free $k$-critical graphs.
He constructed infinitely many $n$-vertex triangle-free 4-critical graphs with at least $\left(\frac{1}{16}-o(1)\right)n^2$ edges,
triangle-free 5-critical graphs with at least $\left(\frac{4}{31}-o(1)\right)n^2$ edges,
and triangle-free $k$-critical graphs with at least $\left(\frac{1}{4}-o(1)\right)n^2$ edges for every $k\geq 6$.
The last bound is asymptotically best possible by Tur\'an's theorem.
He also showed the existence of dense $k$-critical graphs without containing any odd cycle of length at most $\ell$ for any $\ell$,
which is again asymptotically tight for $k\ge 6$.
	
Turning to the upper bound of $f_k(n)$, since any $n$-vertex $k$-critical graph with $n>k$ does not contain $K_k$ as a subgraph,
by Tur\'an's theorem one can easily obtain that $f_k(n)< e(T_{k-1}(n))$ for any $n>k\geq 4$.
Using a characterization of Greenwell and Lov\'asz \cite{GL} for subgraphs of $k$-critical graphs and a classical theorem of Simonovits \cite{S68},
Stiebitz \cite{S} improved this trivial bound in 1987 by showing that
\begin{equation}\label{equ:old}
f_k(n)< e(T_{k-2}(n)) \mbox{ for sufficiently large integer } n.
\end{equation}
It has been 35 years since then and as far as we are aware, this remains the best upper bound.

There is a natural relation between $f_k(n)$ and the problem of determining the maximum number of copies of $K_{k-1}$ in $k$-critical graphs.
Abbott and Zhou \cite{AZ} generalized an earlier result of Stiebitz \cite{S} on 4-critical graphs and showed that for any $k\geq 4$ every $k$-critical graph on $n$ vertices contains at most $n$ copies of $K_{k-1}$.
The bound was further improved in \cite{KS}.
Recently, Gao and Ma \cite{GM} proved a sharp result that for any $n>k\geq 4$, any $k$-critical graph on $n$ vertices contains at most $n-k+3$ copies of $K_{k-1}$.
If we delete one edge for every $K_{k-1}$ in a $k$-critical graph on $n$ vertices,
then this can result in a graph without containing $K_{k-1}$.
Using Tur\'an's theorem and the above result of \cite{GM}, we can derive that
$$f_k(n)\le e(T_{k-2}(n))+n-k+3 \mbox{ for any } n>k\geq 4.$$
	
In this paper, we focus on the upper bound of $f_k(n)$.
Our first result improves the long-standing upper bound \eqref{equ:old} of Stiebitz \cite{S}.
\begin{thm}\label{Thm-general}
For any integer $k\ge 4$ and sufficiently large integers $n$, there exists a constant $c_k\geq \frac{1}{36(k-1)^2}$ such that $f_k(n)\leq e(T_{k-2}(n))-c_k n^2$.
\end{thm}

Our second result considers	4-critical graphs.
A better upper bound for $f_4(n)$ than Theorem~\ref{Thm-general} is obtained in the following.
\begin{thm}\label{Thm-4critical}
For sufficiently large integers $n$, it holds that $f_4(n)<0.164 n^2$.
\end{thm}
	
The proofs of both theorems rely on arguments from extremal graph theory (such as the stability of F{\"u}redi \cite{F}) and a structural lemma (Lemma~\ref{key}) given in the coming section.
Lemma~\ref{key} indicates a partial structure in dense critical graphs (under certain constraints),
which can be witnessed in many classical constructions of dense critical graphs (see the discussion at the beginning of Section 2).
For that, we would like to give a full construction for the well-known {\it Toft graph} (see \cite{Toft70}).
The vertex set of the Toft graph is formed by 4 disjoint sets $A,B,C,D$ with the same odd size,
where $A$ and $D$ are odd cycles, $B$ and $C$ are independent sets, the edges between $B$ and $C$ form a complete bipartite graph, and both of the edges in $(A,B)$ and in $(C,D)$ form perfect matchings.
It is easy to check that the $n$-vertex Toft graph is 4-critical and has $\frac{1}{16}n^2+n$ edges.
We remark that the Toft graph remains the best construction for dense 4-critical graphs.

We use standard notation in graph theory.
Let $\overline{G}$ denote the complement of the graph $G$.
For a vertex $v$ in a graph $G$, let $N_G(v)$ denote the neighborhood of $v$ in $G$, and let $d_G(v):=|N_G(v)|$ denote the degree of $v$ in $G$.
When $G$ is clear from the context, we often drop the subscript.
Let $d(G)$ denote the average degree of the graph $G$.
Also, for any $S\subseteq V(G)$, let $G[S]$ denote the induced subgraph of $G$ on the vertex set $S$.
For any disjoint sets $A,B\subseteq V(G)$, let $G[A,B]$ denote the induced bipartite subgraph of $G$ with bipartition $(A,B)$.

The rest of the paper is organized as follows. 
In Section~2, we prove a lemma which is key for the coming proofs. 
Then we prove Theorem~\ref{Thm-general} in Section~3 and Theorem~\ref{Thm-4critical} in Section~4.

\section{Key lemma}
In this section we prove our key lemma, which roughly says that if a $k$-critical graph $G$ contains certain $t$ copies of $K_{k-2}$ sharing $k-3$ common vertices,
then there exists an ``induced'' matching of size $t$ in $G$ which are connected to these cliques.
This indicates a substructure similar to the Toft graph (and many other examples of $k$-critical graphs).
In particular, it reveals that the structure of $k$-critical graphs cannot be close to the Tur\'an graph $T_{k-2}(n)$ and thus the inequality \eqref{equ:old} should not be tight .
\begin{lem}\label{key}
Let $k\ge 4$ and let $G$ be a $k$-critical graph.
Suppose that $G\left[\{x_1,x_2,\dots, x_{k-3}\}\right]$ forms a copy of $K_{k-3}$ and
there exists a set $W\subseteq N(x_1)\cap \dots\cap N(x_{k-3})\cap N(u)$ for some vertex $u\notin \{x_1,x_2,\dots, x_{k-3}\}$.
Then there exist a set $W'$ and a bijection $\varphi: W\rightarrow W'$ such that $N(\varphi(w))\cap W=\{w\}$ and $N(w)\cap W'=\{\varphi(w)\}$ hold for each $w\in W$.
Moreover, if $|W|\ge3$, then $W$ is an independent set in $G$, and $W'\cap W=\emptyset$.
\end{lem}

\begin{proof}
For each vertex $w\in W$, by deleting the edge $uw$ from the $k$-critical graph $G$, we can get a $(k-1)$-chromatic graph $G'$.
We denote the color classes of $G'$ by $\mathcal{C}_1,\mathcal{C}_2,\dots,\mathcal{C}_{k-1}$.
It is easy to see the vertices $u$ and $w$ are in the same color class.
Since $G[\{x_1,x_2,\dots, x_{k-3},w\}]$ is a $(k-2)$-clique, we can assume $x_1\in\mathcal{C}_1$, $x_2\in\mathcal{C}_2$,\dots, $x_{k-3}\in\mathcal{C}_{k-3}$, and $u,w\in \mathcal{C}_{k-2}$.
The fact $W\subseteq N(x_1)\cap \dots\cap N(x_{k-3})\cap N(u)$ tells us that the set $W\backslash \{w\}$ (if not empty) must be contained in $\mathcal{C}_{k-1}$, and thus $W\backslash \{w\}$ is an independent set in $G$.
We claim $N(w)\cap \mathcal{C}_{k-1}$ must contain a vertex, say $\varphi(w)$.
Since otherwise $\mathcal{C}_{1}, \dots,\mathcal{C}_{k-3},\mathcal{C}_{k-2}-\{w\},\mathcal{C}_{k-1}\cup\{w\}$ can be a $(k-1)$-coloring of $G$, which contradicts the fact that $G$ is $k$-critical.
Besides, $\{\varphi(w)\}\cup(W\backslash \{w\})\subseteq \mathcal{C}_{k-1}$ tells us that $N(\varphi(w))\cap W=\{w\}$.
Now we define $W':=\{\varphi(w):w\in W\}$.
As we have shown that $N(\varphi(w))\cap W=\{w\}$ holds for each $w\in W$, it is easy to see $|W'|=|W|$, $\varphi:W\rightarrow W'$ is a bijection, and $N(w)\cap W'=\{\varphi(w)\}$  holds for each $w\in W$.
		
Moreover, if $|W|\ge3$, then $W$ is an independent set in $G$ (since $W\backslash \{v\}$ is an independent set in $G$ for each vertex $v\in W$).
By the fact that the edges between $W'$ and $W$ precisely form a matching, we can see $W'\cap W=\emptyset$ in this case.
\end{proof}

It would be very interesting to see if this lemma (or its proof) can be extended further.
	
\section{The general case: $k$-critical}
Providing a simple and new proof of the stability for the Tur\'an number ex$(n,K_{r+1})$, F{\"u}redi \cite{F} showed that
if an $n$-vertex graph $G$ is $K_{r+1}$-free and has at least $e(T_{r}(n))-t$ edges where $0\le t<e(T_{r}(n))<n^2$, then there exists a partition $V_1,\dots, V_r$ of $V(G)$ such that $\sum_{i=1}^r e(G[V_i])\le t$.
The proof of \cite{F} (see Corollary 3) also indicates that if the complete $r$-chromatic graph with color classes $V_1,\dots,V_r$ is denoted by $K$,
then $|E(K)\backslash E(G)|\le 2t$ and moreover, $\sum_{i=1}^r \left(|V_i|-n/r\right)^2<4t+o(n^2).$
We summarize in the following lemma.
	
\begin{lem}[F{\"u}redi \cite{F}]\label{stability}
Suppose that $G$ is an $n$-vertex $K_{r+1}$-free graph with $e(G)\ge e(T_{r}(n))-t$ where $0\le t<e(T_{r}(n))<n^2$.
Then there exists a complete $r$-chromatic graph $K:=K(V_1,\dots,V_r)$ with $V(K)=V(G)$ such that
$$|E(K)\backslash E(G)|\le 2t,$$
and $$\sum_{i=1}^r\left(|V_i|-\frac{n}{r}\right)^2<4t+o(n^2).$$
\end{lem}
	
\medskip

We are ready to use Lemmas \ref{key} and \ref{stability} to prove Theorem \ref{Thm-general}.

\begin{proof}[\bf{Proof of Theorem \ref{Thm-general}.}]
Fix $k\ge 4$ and let $C=\frac{1}{36(k-1)^2}$.
Let $G$ be a $k$-critical graph on $n$ vertices with $e(G)> e(T_{k-2}(n))-C n^2$.
In the rest of the proof, we will always assume that $n$ is large enough, and we denote $V(G)$ by $V$ for convenience.
The result in \cite{AZ} tells us the number of copies of $K_{k-1}$ in $G$ is at most $n$.
So by deleting at most $n$ edges in $G$, we obtain a spanning subgraph $G'$ which is $K_{k-1}$-free.
Obviously we have $e(G')> e(T_{k-2}(n))-(C n^2+n)$.
	
With the application of Lemma \ref{stability}, we get a partition $V_1,\dots,V_{k-2}$ of $V$ and a complete $(k-2)$-chromatic graph $K:=K(V_1,\dots,V_{k-2})$ such that $|E(K)\backslash E(G')|\le 2(C n^2+n)$ and
$$\Big||V_i|-\frac{n}{k-2}\Big|<\sqrt{4C n^2+o(n^2)}<\frac{n}{3(k-1)}+o(n) \mbox{ for each }1\le i\le k-2.$$
Without loss of generality, we assume $|V_1|\le\dots\le|V_{k-2}|$.
Thus $|V_{k-2}|\ge n/(k-2)$.
We call the edges in $E(K)\backslash E(G')$ as {\it missing edges}.
And the number of missing edges incident to the vertex $v$ in $K$ is called the {\it missing degree} of $v$.
For each $1\le i\le k-2$, we define $B_i$ to be the set of $\left\lceil\frac{n}{3(k-1)}\right\rceil$ vertices in $V_i$ satisfying that there exists some $m_i$ such that the missing degree of any vertex in $B_i$ is at least $m_i$, and the missing degree of any vertex in $U_i:=V_i-B_i$ is at most $m_i$.
Since there are at most $ 2(C n^2+n)$ missing edges in total, we have
$\sum_{i=1}^{k-2} m_i |B_i| <4(C n^2+n)$, and thus we can get
$$\sum_{i=1}^{k-2} m_i <4(C n^2+n)\Big/\left\lceil\frac{n}{3(k-1)}\right\rceil\le\frac{n}{3(k-1)}+12(k-1).$$
And we can check that for each $1\le i\le k-2$, we have
\begin{align}\label{*}
|U_i|=|V_i|-|B_i|>n/(k-2)-\frac{n}{3(k-1)}-\frac{n}{3(k-1)}-o(n)>\frac{n}{3(k-2)}\ge \sum_{i=1}^{k-2} m_i+\Theta(n).
\end{align}
	
Fix an arbitrary vertex $x_0\in U_{k-2}$ and let $Y:=N_{G'}(x_0)\backslash V_{k-2}$.
It is clear that $$|Y|\ge n-|V_{k-2}|-m_{k-2}.$$
We can find a copy of $K_{k-3}$ in $G'$ on vertices $x_1,x_2,\dots,x_{k-3}$ with $x_i\in U_i\cap Y=U_i\cap N_{G'}(x_0)$ by greedily choosing the vertex $x_i\in U_i\cap N_{G'}(x_0)\cap\dots\cap N_{G'}(x_{i-1})$ for $1\le i\le k-3$ one by one since (\ref{*}) holds for each $1\le i\le k-3$.
Then, since $|U_i|-m_{k-2}\ge|U_i|-\sum_{i=1}^{k-2} m_j>k-2$ holds for each $1\le i\le k-3$ by (\ref{*}), we can find a vertex $u\in U_{i_0}\cap Y$ distinct from $x_1,x_2,\dots,x_{k-3}$, where we choose $i_0$ such that $m_{i_0}=\min\{m_1,\dots,m_{k-3}\}$.
Let $W:=N_{G'}(x_1)\cap\dots\cap N_{G'}(x_{k-3})\cap N_{G'}(u)\cap V_{k-2}$.
We can see $W\ni x_0$, $W\cap Y=\emptyset$, and $$|W|\ge|V_{k-2}|-\sum_{i=1}^{k-3} m_j-m_{i_0}\ge |V_{k-2}|-\left(1+\frac{1}{k-3}\right)\sum_{i=1}^{k-3} m_j.$$
Then by using Lemma \ref{key}, we get a set $W'$ with $|W'|=|W|$ such that $|N_G(w)\cap W'|=1$ for each $w\in W'$, and $|W'\cap W|\le 2$.
Note that all vertices in $Y$ are adjacent to the vertex $x_0\in W$ in $G'\subseteq G$,
so we can see $|W'\cap Y|\le 1$.
	
As $W\cap Y=\emptyset$, $|W'\cap W|\le 2$, $|W'\cap Y|\le1$ and $|W'|=|W|$, we get $n\ge |W\cup Y\cup W'|\ge 2|W|+|Y|-3$.
Thus
$$2|W|+|Y|\le n+3.$$
But on the other hand, we can check that
\begin{align*}
	2|W|+|Y|&\ge 2\left(|V_{k-2}|-\left(1+\frac{1}{k-3}\right)\sum_{j=1}^{k-3} m_j\right)+ \left(n-|V_{k-2}|-m_{k-2}\right)\\
	&\ge n+|V_{k-2}|-2\left(1+\frac{1}{k-3}\right)\sum_{j=1}^{k-2} m_j\\
	&\ge n+\frac{n}{k-2}-2\left(1+\frac{1}{k-3}\right)\left(\frac{n}{3(k-1)}+12(k-1)\right)>n+3.
\end{align*}
This derives a contradiction.
So we have $f_k(n)\le e(T_{k-2}(n))-C n^2$ for $n$ sufficiently large.
\end{proof}

We would like to remark that the above proof relies on the existence of $K_{k-2}$.
(Recall that in Lemma \ref{key}, $G[\{w,x_1,x_2,\dots,x_{k-3}\}]$ forms a copy of $K_{k-2}$ for each vertex $w\in W$.)
So using this approach, we will not be able to improve the upper bound to the following $$e(G)\leq \ex(n,K_{k-2})=e(T_{k-3}(n))\leq e(T_{k-2}(n))-\frac{n^2}{2(k-2)(k-3)};$$
that says, we are not able to obtain a constant $c_k$ better than the order of magnitude $k^{-2}$.

\section{The 4-critical case}
In this section we consider 4-critical graphs and prove Theorem \ref{Thm-4critical}.

Before presenting the proof of Theorem \ref{Thm-4critical},
we like to give a short proof of a slightly weaker bound (see Theorem \ref{Thm-4critical'})
than Theorem \ref{Thm-4critical} to illustrate the proof ideas.
In doing this, we study certain local structure based on 2-paths (i.e., a path of length two) in the proof of Theorem \ref{Thm-4critical'},
while we consider 4-cycles (i.e., a cycle of length four) in replace of 2-paths in the proof of Theorem \ref{Thm-4critical}.

\subsection{A weaker upper bound}
We first show the following result.
\begin{thm}\label{Thm-4critical'}
For any integer $n\geq 4$, it holds that $f_4(n)<\frac{1}{6} n^2+10 n\leq 0.167n^2+10n$.
\end{thm}

We also need two lemmas as follows.
For a graph $G$, we denote $t(G)$ to be the number of triangles in $G$.
For a vertex $v$, let $t_G(v)$ be the number of triangles containing the vertex $v$ in $G$. When $G$ is clear, we often drop the subscript.

\begin{lem}\label{lem1}
Suppose $G$ has at most $n$ triangles and minimum degree at least 3. Then $G$ contains a 2-path $xyz$ such that $$d(x)+d(y)+d(z)-3t(x)-3t(z)\ge\frac{6e(G)}{n}-\frac{9n^2}{e(G)}.$$
\end{lem}
\begin{proof}
For some vertex $v\in V(G)$, write $N(v)=\{v_1,v_2,\dots,v_t\}$ for some $t\geq 3$.
Let $$\mathcal{P}_v:=\{v_1 v v_2,\dots, v_{t-1} v v_{t},v_{t} v v_1\}$$ be a family of 2-paths with center $v$.
We have $|\mathcal{P}_v|=d(v)$,
and $$\sum_{xyz\in\mathcal{P}_v} \left(d(x)+d(y)+d(z)\right)={d(v)}^2+2\sum_{u\in N(v)} d(u),$$
$$\sum_{xyz\in\mathcal{P}_v} \left(t(x)+t(z)\right)=2\sum_{u\in N(v)} t(u).$$
Then let $\mathcal{P}:=\bigcup_{v\in V(G)}\mathcal{P}_v$.
We have $$|\mathcal{P}|=\sum_{v\in V(G)}d(v)=2e(G).$$
Using Jensen's inequality, we get
\begin{align*}
\sum_{xyz\in\mathcal{P}} (d(x)+d(y)+d(z))
&=\sum_{v\in V(G)}{d(v)}^2+2\sum_{v\in V(G),u\in N(v)} d(u)=\sum_{v\in V(G)}{d(v)}^2+2\sum_{u\in V(G),v\in N(u)} d(u)\\
&=\sum_{v\in V(G)}{d(v)}^2+2\sum_{u\in V(G)}{d(u)}^2=3\sum_{v\in V(G)}{d(v)}^2\ge 12 e(G)^2/n.
\end{align*}
Since every vertex in $G$ has degree at most $n-1$ and $\sum_{u\in V(G)} t(u)=3t(G)\le 3n$, we get
$$\sum_{xyz\in\mathcal{P}} (t(x)+t(z))=2\sum_{v\in V(G)} \sum_{u\in N(v)} t(u)=2\sum_{u\in V(G)}d(u)t(u)
\le 2n\sum_{u\in V(G)} t(u)
\le 6n^2.$$
So by picking a 2-path $xyz$ in $\mathcal{P}$ uniformly and randomly, we   see
$$\mathbb{E}[d(x)+d(y)+d(z)-3t(x)-3t(z)]\ge \frac{12 e(G)^2/n-18n^2}{|\mathcal{P}|}=
\frac{6e(G)}{n}-\frac{9n^2}{e(G)}.$$
Thus we can find a 2-path $xyz$ as desired.
\end{proof}
	
\begin{lem}\label{lem2}
For any 2-path $xyz$ in a 4-critical graph $G$, we have $$d(x)+d(y)+d(z)-3t(x)-3t(z)\le n+1.$$
\end{lem}
\begin{proof}
Let $X:=N(x)$, $Y:=N(y)$, $Z:=N(z)$, and $W:=X\cap Z$.
If $u\in X\cap Y$, $uxy$ is a triangle. So $|X\cap Y|\le t(x)$. Similarly, $|Z\cap Y|\le t(z)$.
Then we have
$$|X\cup Y\cup Z|\ge|X|+|Y|+|Z|-|X\cap Y|-|Z\cap Y|-|X\cap Z|\ge d(x)+d(y)+d(z)-t(x)-t(z)-|W|.$$
		
By Lemma \ref{key}, we can find a set $W'\subseteq V(G)$ and a bijection $\varphi: W\rightarrow W'$ such that $W'=\{\varphi(w):w\in W'\}$, and for each $w\in W$, we have both $N(\varphi(w))\cap W=\{w\}$ and $N(w)\cap W'=\{\varphi(w)\}$.
		
We consider the size of $W'\cap  (X\cup Y\cup Z)$.
Since both $N(\varphi(w))\cap W=\{w\}$ and $N(w)\cap W'=\{\varphi(w)\}$ hold for each $w\in W$, and we know $y\in W$, we can see $|W'\cap Y|\le|W'\cap N(y)|\le1$.
Suppose $v'\in W'\cap X$. There is a vertex $v\in W$ such that $vv'$ is an edge. Then we see $xvv'$ is a triangle.  So $|W'\cap X|\le 2t(x)$. Similarly,   $|W'\cap Z|\le 2t(z)$.  Totally, we have
$$|W'\cap  (X\cup Y\cup Z)|\le |W'\cap X|+|W'\cap Y|+|W'\cap Z|\le 2t(x)+2t(z)+1.$$
Finally, we get
\begin{align*}
n&\ge |X\cup Y\cup Z\cup W'|=|X\cup Y\cup Z\cup W'|+|W'|-|W'\cap  (X\cup Y\cup Z)|\\
&=\left(d(x)+d(y)+d(z)-t(x)-t(z)-|W|\right)+|W|-(2t(x)+2t(z)+1)\\
&=d(x)+d(y)+d(z)-3t(x)-3t(z)-1,
\end{align*}
completing the proof of this lemma.
\end{proof}
	
Now we can finish the proof of this subsection.
	
\begin{proof}[\bf Proof of Theorem \ref{Thm-4critical'}]
Let $G$ be an $n$-vertex 4-critical graph. It is easy to see that the minimum degree of $G$ is at least 3.
By the result in \cite{S}, $G$ contains at most $n$ copies of triangles.
Applying Lemma \ref{lem1}, we can get a 2-path $xyz$ with $$d(x)+d(y)+d(z)-3t(x)-3t(z)\ge\frac{6e(G)}{n}-\frac{9n^2}{e(G)}.$$
Together with Lemma \ref{lem2}, we have
$$\frac{6e(G)}{n}-\frac{9n^2}{e(G)}\le n+1.$$
This implies that $e(G)< n^2/6+10 n.$
\end{proof}
	
\subsection{The proof of Theorem \ref{Thm-4critical}}\label{4}
To show Theorem \ref{Thm-4critical}, we need some new lemmas.
The coming lemma can be easily obtained by averaging, which says that every graph contains an edge such that the sum of the degrees of its two endpoints is at least twice the average degree of the graph.
\begin{lem}\label{lem-edge}
Any graph $G$ contains an edge $xy$ such that $$d(x)+d(y)\ge 2d(G).$$
\end{lem}
\begin{proof}
By Jensen's inequality, we can get
$$\sum_{xy\in E}\left(d(x)+d(y)\right)=\sum_{v\in V}d(v)^2\ge n d(G)^2.$$
Note that $|E|=\left(n d(G)\right)/2$.
Thus there exists an edge $xy\in E$ such that
$$d(x)+d(y)\ge \frac{ n d(G)^2}{\left(n d(G)\right)/2}=2d(G),$$
proving the lemma.
\end{proof}
	
We now give the following lemma about 4-cycles, which can be viewed as a generalization of the previous lemma.
Recall the well-known result of Reiman \cite{R} that any $n$-vertex graph without containing 4-cycles has at most $\frac n 4 (1+\sqrt{4n-3})<n^{\frac32}$ edges.
	
\begin{lem}\label{lem-4cycle}
Any $n$-vertex graph $G$ with $e(G)>\frac n 4 (1+\sqrt{4n-3})$ contains a 4-cycle $v_1v_2v_3v_4$ satisfying that $$d(v_1)+d(v_2)+d(v_3)+d(v_4)\ge 4d(G)-O(n^{\frac34}).$$
\end{lem}
\begin{proof}
Fix $\epsilon:=9n^{-\frac14}$.
Note that $G$ must contain 4-cycles by the result of Reiman \cite{R}.
Suppose to the contrary that any 4-cycle $v_1v_2v_3v_4$ in $G$ satisfies $d(v_1)+d(v_2)+d(v_3)+d(v_4)< 4d(G)-4\epsilon n.$
Let $A:=\{v\in V:d(v)<d(G)\}$ and $B:=\{v\in V:d(v)\ge d(G)\}$.
Then $A\cup B$ forms a partition of $V(G)$ such that $G[B]$ does not contain any 4-cycle.
		
For each $1\le i\le d(G)/\epsilon n$, let $A_i:=\{v\in V:d(G)-i \epsilon n\le d(v)< d(G)-(i-1)\epsilon n\}$. Then these $A_i$'s form a partition of $A$.
For each $1\le i\le\left(n-d(G)\right)/\epsilon n$, let $B_i:=\{v\in V:d(G)+(i-1) \epsilon n\le d(v)< d(G)+i\epsilon n\}$. Then these $B_i$'s form a partition of $B$.
It is not hard to check that $G[A_1]$ does not contain any 4-cycle,
and for each $1\le i\le\left(n-d(G)\right)/\epsilon n$, $G\left[\bigsqcup_{j=1}^{i+1}A_j,B_i\right]$ does not contain any 4-cycle.
		
We delete all edges in $G[B]$, $G[A_1]$ and $G\left[\bigsqcup_{j=1}^{i+1}A_j,B_i\right]$ for each $1\leq i\le\left(n-d(G)\right)/\epsilon n$ to get a spanning subgraph $G'$ of $G$.
By the result of Reiman \cite{R}, we can obtain $$e(G')\ge e(G)-\left(2+\left(n-d(G)\right)/\epsilon n\right) n^{\frac32}\ge e(G)-2n^{\frac32}-\frac19 n^{\frac74}\ge e(G)-\frac{19}{9}n^{\frac74}.$$
Thus we have $$d(G')\ge d(G)-\frac{38}{9}n^{\frac34}.$$
Note that any edge of $G'$ is either contained in $A$, or between $A_j$ and $B_i$ for some $j\ge i+2$; moreover, we have $e(G'[A_1])=0$.
Thus, as $n$ is large enough, it is easy to check that for any edge $xy$ in $G'$, $$d_{G'}(x)+d_{G'}(y)<2d(G)-\epsilon n=2d(G)-9 n^{\frac34}<2d(G').$$
This contradicts Lemma \ref{lem-edge}, thus proving Lemma \ref{lem-4cycle}.	
\end{proof}
	
The following lemma is derived from Lemma \ref{key}, which provides an essential structure to the proof of Theorem \ref{Thm-4critical}.
\begin{lem}\label{key-4critical}
Let $G$ be a 4-critical graph. Suppose $v_1v_2v_3v_4$ is a 4-cycle in $G$,
and $V_1,V_2,V_3,V_4$ are four sets such that $\{v_2,v_4\}\subseteq V_1\subseteq N(v_1)$, $\{v_1,v_3\}\subseteq V_2\subseteq N(v_2)$, $\{v_2,v_4\}\subseteq V_3\subseteq N(v_3)$, and $\{v_1,v_3\}\subseteq V_4\subseteq N(v_4)$.
Let $X=V_1\cap V_3$ and $Y=V_2\cap V_4$.
Then there exist sets $X''$ and $Y''$ such that
\begin{itemize}
			\item $X''\cap\left(V_1\cup V_2\cup V_3\cup V_4\right)=\emptyset=Y''\cap\left(V_1\cup V_2\cup V_3\cup V_4\right)$,
			\item $e(G[X'',X])\le |X|$ and $e(G[Y'',Y])\le |Y|$, and
			\item $|X''|\ge|X|-2t_G(v_1)-2t_G(v_3)-2$ and $|Y''|\ge|Y|-2t_G(v_2)-2t_G(v_4)-2$.
\end{itemize}
\end{lem}
\begin{proof}
As $X\subseteq N(v_1)\cap N(v_3)$, by Lemma \ref{key} for $k=4$, there exists a set $X'\subseteq V(G)$ and a bijection $\varphi: X\rightarrow X'$ such that $X'=\{\varphi(x):x\in X\}$, and for each $x\in X$, we have both $N(\varphi(x))\cap X=\{x\}$ and $N(x)\cap X'=\{\varphi(x)\}$.
We define $X'':=X'\backslash\left(V_1\cup V_2\cup V_3\cup V_4\right)$, then obviously $X''\cap\left(V_1\cup V_2\cup V_3\cup V_4\right)=\emptyset$ and $e(G[X'',X])\le |X|$.
		
As $Y\subseteq N(v_2)\cap N(v_4)$, by Lemma \ref{key} for $k=4$, there exists a set $Y'\subseteq V(G)$ and a bijection $\phi: Y\rightarrow Y'$ such that $Y'=\{\phi(y):y\in Y\}$, and for each $y\in Y$, we have both $N(\phi(y))\cap Y=\{y\}$ and $N(y)\cap Y'=\{\phi(y)\}$.
We define $Y'':=Y'\backslash\left(V_1\cup V_2\cup V_3\cup V_4\right)$, then obviously $Y''\cap\left(V_1\cup V_2\cup V_3\cup V_4\right)=\emptyset$ and $e(G[Y'',Y])\le |Y|$.
		
Then we want to show the last property.
		
All vertices in $V_2$ are adjacent to the vertex $v_2\in X$.
Then we have $|X'\cap V_2|\le 1$ since $|N(x)\cap X'|=1$ for each $x\in X$. Similarly, we have $|X'\cap V_4|\le 1$, $|Y'\cap V_1|\le 1$, and $|Y'\cap V_3|\le 1$.
		
All vertices in $V_1$ are adjacent to the vertex $v_1$.
Since each vertex in $X'$ has a neighbor in $X\subseteq N(v_1)$, we can check that $|X'\cap V_1|\le 2 t(v_1)$. Similarly, we have $|X'\cap V_3|\le 2 t(v_3)$, $|Y'\cap V_2|\le 2 t(v_2)$, $|Y'\cap V_4|\le 2 t(v_4)$.
Therefore, $$|X''|= |X'|-|X'\cap\left(V_1\cup V_2\cup V_3\cup V_4\right)|\ge |X|-2t(v_1)-2t(v_3)-2,$$
and $$|Y''|= |Y'|-|Y'\cap\left(V_1\cup V_2\cup V_3\cup V_4\right)|\ge |Y|-2t(v_2)-2t(v_4)-2,$$
completing the proof.		
\end{proof}
	
Now we are ready to prove Theorem \ref{Thm-4critical}.
	
\begin{proof}[\bf Proof of Theorem \ref{Thm-4critical}]
Throughout this proof, we assume that $n$ is sufficiently large, and the subscripts of the notations such as $v_i$'s and $V_i$'s are under module $4$.
Suppose by contradiction that there exists an $n$-vertex 4-critical graph $G$ with $e(G)\ge 0.164 n^2$.
By the result in \cite{S}, $G$ contains at most $n$ copies of triangles.
Let $V_0:=\{v\in V(G): t_G(v)\ge \sqrt n\}$.
Then clearly we have $|V_0|<3\sqrt n$.
Let $G':=G[V(G)-V_0]$.
It is not hard to see $e(G')\ge e(G)-n |V_0|>e(G)-3n^{\frac32}\ge 0.164 n^2-o(n^2)$.
Note that $t(G')\le t(G)\le n$.
Therefore, by deleting at most $n$ edges from $G'$, we can get a subgraph $G''\subseteq G'$ 
such that $t(G'')=0$, $e(G'')\ge e(G')-n\ge 0.164 n^2-o(n^2)$, and $t_{G}(v)<\sqrt n$ for each $v\in V(G'')=V(G)-V_0$.
By applying Lemma \ref{lem-4cycle} to $G''$, we can get a 4-cycle $v_1v_2v_3v_4$ in $G''$ such that
\begin{align}\label{V-sum}
|V_1|+|V_2|+|V_3|+|V_4|\ge 8e(G'')/n-o(n)\ge 1.312 n-o(n),
\end{align}
where $V_i:=N_{G''}(v_i)$ for each $1\le i\le 4$.
Note that for each $1\le i\le 4$, every vertex in $V_i\cap V_{i+1}$ must form a triangle with the vertices $v_i, v_{i+1}$ in $G''$, which contradicts the fact $t(G'')=0$.
So it is clear that $$V_i \cap V_{i+1}=\emptyset \mbox{ for each }1\le i\le4.$$
Also it is easy to check that $\{v_{i-1},v_{i+1}\}\subseteq V_i\subseteq N_{G}(v_i)$ for each $1\le i\le4$.
Define $X=V_1\cap V_3$ and $Y=V_2\cap V_4$.
Applying Lemma \ref{key-4critical}, we can get two sets $X'',Y''$ satisfying the three properties of Lemma \ref{key-4critical}.
Note that $X''$ and $Y''$ are disjoint from $V_1\cup V_2\cup V_3\cup V_4$, $V_1\cap V_3=X$, $V_2\cap V_4=Y$,
and $V_i\cap V_{i+1}=\emptyset$ for each $1\le i\le 4$.
So we can see that $$|V_1|+|V_2|+|V_3|+|V_4|-|X|-|Y|+|X''\cup Y''|\le n.$$
Besides, by using the last property in Lemma \ref{key-4critical}, we have
\begin{align}\label{''}
|X''\cup Y''|\ge \max\{|X''|,|Y''|\}\ge \frac{|X''|+|Y''|}{2}\ge\frac {|X|+|Y|}{2}-O(\sqrt n).
\end{align}
The above two inequalities tell us that
\begin{align}\label{X+Y}
\frac{|X|+|Y|}{2}\ge |V_1|+|V_2|+|V_3|+|V_4|-n-O(\sqrt n)\ge  0.312 n-o(n).
\end{align}
		
Then we consider the non-edges of the graph $G$, i.e., the edges of the graph $\overline{G}$.
First, since $V_i= N_{G''}(v_i)\subseteq N_{G}(v_i)$ and $v_i\in V(G'')$, we can see $e(G[V_i])\le t_{G}(v_i)\le \sqrt n$ for each $1\le i\le 4$.
So $$e(\overline{G}[V_i])\ge \binom{|V_i|}{2}-o(n^2)=\frac12 |V_i|^2-o(n^2) \mbox{ for each }1\le i\le4.$$
Thus by noting $V_1\cap V_3=X$, $V_2\cap V_4=Y$, and $V_i\cap V_{i+1}=\emptyset$ for each $1\le i\le 4$, we can get
$$\left|\bigcup_{i=1}^{4}E(\overline{G}[V_i])\right|\ge\sum_{i=1}^{4}e(\overline{G}[V_i])-\binom{|X|}{2}-\binom{|Y|}{2}\ge\frac12 \left(\sum_{i=1}^{4}|V_i|^2-|X|^2-|Y|^2\right)-o(n^2).$$
Next, since $G$ can be made triangle-free by deleting at most $n$ edges and any $n$-vertex triangle-free graph has at most $\frac14 n^2$ edges, we can see $e(G[X''\cup Y''])\le \frac14 |X''\cup Y''|^2+n$, and thus $$e(\overline{G}[X''\cup Y''])\ge \frac14 |X''\cup Y''|^2-o(n^2).$$
By the properties on $X'', Y''$ we derive from Lemma \ref{key-4critical}, we can obtain
$$e(\overline{G}[X'',X])\ge |X''| |X|-|X|\ge |X|^2-o(n^2),$$
$$e(\overline{G}[Y'',Y])\ge|Y''||Y|-|Y|\ge |Y|^2-o(n^2).$$
By the above three inequalities, we can deduce that
\begin{align*}
e(G)&= \binom{n}{2}-e(\overline{G})\le \binom{n}{2}-\frac12 \left(\sum_{i=1}^{4}|V_i|^2-|X|^2-|Y|^2\right)-\frac14 |X''\cup Y''|^2-|X|^2-|Y|^2+o(n^2)\\
&\le \frac12 n^2-\frac18 \left(|V_1|+|V_2|+|V_3|+|V_4|\right)^2-\frac14\left(\frac{|X|+|Y|}{2}\right)^2-\left(\frac{|X|+|Y|}{2}\right)^2+o(n^2)\\
&\le \frac12 n^2-\frac18 (1.312n)^2-\frac54 (0.312n)^2+o(n^2)<0.1632n^2+o(n^2),
\end{align*}
where the second inequality comes from the inequality (\ref{''}),
and the third inequality comes from the inequalities (\ref{V-sum}) and (\ref{X+Y}).
This contradicts the assumption that $e(G)\ge 0.164n^2,$ completing the proof of Theorem \ref{Thm-4critical}.	
\end{proof}

Our understanding for the functions $f_k(n)$ is generally poor, and it is not even known if 
\begin{align}\label{equ:f45}
\mbox{$f_4(n)<f_5(n)$ holds for sufficiently large integers $n$.}
\end{align}
So it seems to be a natural next step to pursue the question that if $f_4(n)\leq cn^2$ holds for some constant $c<\frac{4}{31}$ and sufficiently large $n$. Note that if this is true, then it would imply \eqref{equ:f45}.

\bigskip

\noindent{\bf Acknowledgement.} The authors would like to thank Prof. Alexandr Kostochka for many valuable comments on a preliminary version of the manuscript.
	
	\bibliographystyle{unsrt}

\begin{thebibliography}{99}
		
		\bibitem{AZ}
		H. L. Abbott and B. Zhou,
		\newblock{On a conjecture of Gallai concerning complete subgraphs of $k$-critical graphs},
		\newblock{\emph{Discrete Math.}} \textbf{100} (1992), 223--228.
		
		\bibitem{D}
		G. A. Dirac,
		\newblock{A property of 4-chromatic graphs and some remarks on critical graphs},
		\newblock{\emph{J. London Math. Soc.}} \textbf{27} (1952), 85--92.
		
		\bibitem{F}
		Z. F{\"u}redi,
		\newblock{A proof of the stability of extremal graphs, Simonovits' stability from Szemer\'edi's regularity},
		\newblock{\emph{J. Combin. Theory Ser. B}} \textbf{115} (2015), 66--71.
		
		\bibitem{GM}
		J. Gao and J. Ma,
		\newblock{Tight bounds towards a conjecture of Gallai},
		\newblock{\emph{Combinatorica}}, to appear.

        \bibitem{GL}
        D. Greenwell and L. Lov\'asz, 
        \newblock{Applications of product colouring},
        \newblock{\emph{Aeta Math. Acad. Sci. Hungar.}} \textbf{25} (1974), 335--340.
		
		\bibitem{JT}
		T. R. Jensen and B. Toft,
		\newblock{Graph coloring problems},
		John Wiley \& Sons, 2011.
		
		\bibitem{KS}
		A. E. K\'ezdy and H. S. Snevily,
		\newblock{On extensions of a conjecture of Gallai},
		\newblock{\emph{J. Combin. Theory Ser. B}} \textbf{70} (1997), 317--324.
		
		
		\bibitem{P}
		W. Pegden,
		\newblock{Critical graphs without triangles: an optimum density construction},
		\newblock{\emph{Combinatorica}} \textbf{33} (2013), 495--512.
		
		\bibitem{R}
		I. Reiman,
		\newblock{{\"U}ber ein Problem von K. Zarankiewicz},
		\newblock{\emph{Acta Math. Acad. Sci. Hungar.}} \textbf{9} (1958), 269--273.
	

        \bibitem{S68} 
        M. Simonovits, 
        \newblock{A method for solving extremal problems in graph theory, stability problems},
        in: \newblock{\emph{Theory of Graphs, Proc. Colloq. Tihany, 1966}}, Academic Press, New York, 1968.


	
		\bibitem{S}
		M. Stiebitz,
		\newblock{Subgraphs of colour-critical graphs},
		\newblock{\emph{Combinatorica}} \textbf{7} (1987), 303--312.	
		
		\bibitem{Toft70}
		B. Toft,
		\newblock{On the maximal number of edges of critical $k$-chromatic graphs},
		\newblock{\emph{Studia Sci. Math. Hungar.}} \textbf{5} (1970), 461--470.
		
		
	\end{thebibliography}
	
\end{document}